\documentclass{amsart}
\usepackage{fancyhdr}
\usepackage{amsmath}
\usepackage{amssymb}
\usepackage{amsthm}
\usepackage{bbm}
\usepackage{verbatim}
\usepackage{latexsym}
\usepackage{color}
\usepackage{mathdots}
\usepackage{xypic}
\usepackage[margin=1.25in]{geometry}
\usepackage{enumitem}
\usepackage{stmaryrd}
\usepackage{esint}
\usepackage{cite}
\usepackage{hyperref}
\usepackage{mathtools}

\newcommand{\bb}[1]{\mathbb{#1}}
\newcommand{\cc}[1]{\mathcal{#1}}

\numberwithin{equation}{section}
\newtheorem{theorem}{Theorem}
\newtheorem{corollary}[theorem]{Corollary}
\newtheorem{lemma}[theorem]{Lemma}
\newtheorem{prop}[theorem]{Proposition}
\numberwithin{theorem}{section}
\newenvironment{sproof}{%
  \proof}{\endproof}

\begin{document}

\title{The Hodge Star Operator and the Beltrami Equation}
\author{Eden Prywes}
\maketitle
\begin{abstract}
An essentially unique homeomorphic solution to the Beltrami equation was found in the 1960s using the theory of Calder\'{o}n-Zygmund and singular integral operators in $L^p(\bb C)$.  We will present an alternative method to solve the Beltrami equation using the Hodge star operator and standard elliptic PDE theory.  We will also discuss a different method to prove the regularity of the solution.  This approach is partially based on work by Dittmar \cite{dittmar}.
\end{abstract}

\section{Introduction}
A quasiconformal map is an orientation-preserving homeomorphism $f\colon \bb C \to \bb C$, such that $f \in W^{1,2}_{\text{loc}}(\bb C)$ and there exists $K \ge 1$ with
\begin{align*}
|Df(z)|^2 \le K J_f(z)
\end{align*}
for a.e.\ $z \in \bb C$. Here, $Df$ is the derivative of $f$ and $J_f$ is the Jacobian of $f$.  The above equation can be rewritten as
\begin{align*}
|\partial_{\bar z}f|\le k |\partial_z f|
\end{align*}
where $k < 1$.  In other words, $f$ solves the following differential equation, often called the Beltrami Equation,
\begin{align}
\partial_{\bar z}f = \mu \partial_{z} f, \label{beq}
\end{align}
with $\mu \in L^\infty(\bb C)$ and $\|\mu\|_\infty \le k$.

The main goal of this paper is to prove the following theorem.
\begin{theorem}\label{mainthm}
Let $\mu \colon \bb C \to \bb C$ be measurable with compact support, $\|\mu\|_\infty = k < 1$.  Then there exists an orientation-preserving homeomorphism $\Phi \colon \bb C \to \bb C$ such that $\Phi \in W^{1,2}_{\text{loc}}(\bb C)$ and $\Phi$ solves $\eqref{beq}$ a.e. Additionally, the map $\Phi$ can be chosen so that $\Phi(0) = 0$ and $\Phi(1) = 1$.  In this case it is unique.
\end{theorem}
This theorem will follow from Theorem \ref{mainthm2} which proves the same conclusions with the stronger assumption that $\mu$ is $C^\infty$-smooth.  The hard part of the proof is showing that the theorem is true for the smooth case.  Once this is known, a short argument using well-known facts from quasiconformal theory gives Theorem \ref{mainthm}.  This will be presented in Section \ref{epde}.
The same result for $\mu$ without compact support follows easily from this theorem (see \cite[Ch. 5, Thm. 3]{ahlfors}). 

The first proof for Theorem $\ref{mainthm}$ was given by Morrey in \cite{morrey}.  Morrey also proved that the solutions are H\"older continuous. 
The most well known proof for a solution of \eqref{beq} comes from Bojarski in \cite{bojarski}.  The author of \cite{bojarski} employs singular integral operators and the theory of Calder\'{o}n-Zygmund operators.  

In order to give context to the approach used in this paper it is useful to review the proof from \cite{bojarski}.  An outline of it is as follows.  
Let $\mu \in C_c^\infty(\bb C)$.  Define the following two singular integral operators, the Cauchy transform,
\begin{align*}
Tf(z) = \frac{1}{\pi}\int_{\bb C} \frac{f(w)}{z-w} dw,
\end{align*}
and the Beurling-Ahlfors transform,
\begin{align*}
Hf(z) = \lim_{\epsilon \to 0}\frac{1}{\pi} \int_{\bb C \setminus B(z,\epsilon)} \frac{f(w)}{(w-z)^2}dw.
\end{align*}
One initially defines these operators for $f\in C_c^\infty(\bb C)$.  A version of $T$ can then be extended to $f \in L^p(\bb C) , p>2$.  In this case $Tf$ is $(1-2/p)$-H\"older. 
When $f \in C^1(\bb C)$ we also have that $T$ acts as an inverse to the $\partial_{\bar z}$ operator and $H$ maps $\partial_{\bar z} f$ to $\partial_z f$.  Also note that $\partial_z(Tf) = Hf$.

The theory of Calder\'{o}n-Zygmund operators is then used to extend $H$ to a bounded map on $L^p(\bb C), p>1$ and 
 $H$ is an isometry on $L^2(\bb C)$ (see \cite[Ch. 2]{stein}).  This then gives that as $p \to 2$, the operator norm of $H$ will approach $1$.  Since $|\mu| \le k < 1$, we get that $\|H(\mu\cdot)\|_{\text{op}} < 1$ as an operator on $L^p(\bb C)$ for $p$ sufficiently close to $2$.  Therefore $J = (\text{Id} - H\mu)^{-1}$ exists.
Let
\begin{align*}
\Phi(z) &= T(\mu + \mu J(H\mu)) + z.\\
\intertext{\noindent Note that $J = \sum_{n=0}^\infty (H\mu)^n$, so}
\partial_{\bar z} \Phi(z) &= \mu  + \mu \sum_{n=0}^\infty (H\mu)^{n+1}\\
\intertext{since $\partial_{\bar z}Tf = f$.  Also}\\
\mu \partial_z \Phi(z) &= \mu + \mu H\mu + \mu H \mu \sum_{n=0}^\infty (H\mu)^{n+1}  = \mu  + \mu \sum_{n=0}^\infty (H\mu)^{n+1}.
\end{align*}
So $\Phi$ is satisfies $\eqref{beq}$.

From the above computation it is not clear which properties $\Phi$ has.  We would like for $\Phi$ to be a homeomorphism but it is not evident here that it is even continuous.
Suppose $\mu \in C_c^\infty(\mathbb C)$, then without the Calder\'{o}n-Zygmund theory it is not hard to show that $H$ is an isometry on $L^2(\bb C)$ and that $H\mu \in L^2(\bb C)\cap C^\infty(\bb C)$.  So $J = (\operatorname{Id} - H\mu)^{-1}$ exists on $L^2(\bb C)$.  Additionally, $\mu + \mu J(H\mu) \in L^2(\bb C)$ and has compact support.  This is not enough to ensure that $T(\mu + \mu J(H\mu))$ is well-defined.
In other words, we run into difficulty even when $\mu$ is a very well behaved function.

By using the Calder\'{o}n-Zygmund theory,  we see that 
$H$ is an operator on $L^p(\bb C)$ for $p >2$, which gives that $\mu + \mu J(H\mu) \in L^p(\bb C)$.  This is sufficient to imply that $T(\mu + \mu J(H\mu))$ is H\"{o}lder continuous by straightforward arguments.
One of the purposes of the proof in this paper is to avoid the Calder\'on-Zygmund theory entirely and try to achieve continuity using only $L^2$-methods.

Even if we prove that $\Phi$ is continuous using the approach outlined above it is not clear that $\Phi$ is a homeomorphism.  
To try to show that the solution does have this property one first makes the assumption that $J_\Phi$ is nonzero.  Specifically, we write $\partial_z \Phi = e^\Psi$.  This gives a nonzero Jacobian,
\begin{align*}
J_\Phi &= |\partial_z \Phi|^2 - |\partial_{\bar z} \Phi|^2 \\
 &= |\partial_z \Phi|^2 (1-|\mu|^2) = |e^\Psi|^2(  1-|\mu|^2) > 0. \\
\intertext{If $\partial_z \Phi = e^\Psi$ and $\partial_{\bar z} \Phi =\mu e^\Psi$, then}
\partial_{\bar z} e^\Psi &= \partial_z (\mu e^\Psi).\\
\intertext{This implies that}
e^\Psi \partial_{\bar z} \Psi &= e^{\Psi}(\mu \partial_z \Psi + \partial_z \mu).
\end{align*}
If we cancel the exponential term we get an inhomogeneous Beltrami equation for $\Psi$,
\begin{align*}
\partial_{\bar z}\Psi - \mu \partial_{z} \Psi = \partial_z \mu.
\end{align*}
This is the motivation for studying the inhomogeneous version of \eqref{beq}.  If a solution exists then it will lead to the existence of a solution for \eqref{beq} that is a homeomorphism.
The proof that the solution is a homeomorphism is explained in detail in Section \ref{epde}.

Other proofs for Theorem \ref{mainthm} can be found in the literature.  One such proof comes from \cite{glutsyuk}.  The author of \cite{glutsyuk} uses a homotopy method to solve \eqref{beq} on the torus and then generalizes to $\bb C$.  The proof relies on similar integral operators as in \cite{bojarski}.  However, the proof only applies the operators on functions in $L^2(\bb C)$.  So the method does not need to use Calder\'{o}n-Zygmund theory.  Another proof is presented in \cite{hubbard}.

Dittmar studies the same problem as well (see \cite{dittmar} and \cite{ligocka}).  He converts \eqref{beq} into a elliptic PDE and then solves the equation using the standard theory.  
It is then a well known fact that functions that satisfy uniformly elliptic PDEs with smooth data are smooth.  This fact was initially proved by J. Schauder in \cite{schauder}.  For a thorough discussion of this see \cite[Ch. 6]{trudinger}.

The method to proving Theorem \ref{mainthm} in this paper follows Dittmar's approach.  The proof in \cite{dittmar} is not widely known.  The author of \cite{dittmar} uses real variable notation which yields long and difficult to understand differential equations. To avoid this, in this paper we try to tie this approach to the geometric aspects of the problem.  The Beltrami equation describes the conformal structure for a given Riemannian metric.  When viewed through this lens the equations arise naturally.  Additionally, due to this motivation, we always use complex variable notation, which is more natural in this setting.

Historically, the Beltrami equation arises when studying the conformal geometry of a surface.  The function $\mu$ from \eqref{beq} defines a Riemannian metric $|dz + \mu d\bar z|^2$ on $\bb C$.  In fact any metric on $\bb C$ can be written in coordinates as
\begin{align*}
g = \nu|dz + \mu d\bar z|^2,
\end{align*}
where $\nu$ is a positive real function and $\|\mu\|_\infty < 1$.  This shows that $\mu$ encodes a conformal structure on $\bb C$.  Solving \eqref{beq} gives a map that changes the structure given by $\mu$ to the standard conformal structure.  To exploit the relationship between equation $\eqref{beq}$ and the geometry of $\bb C$ we use the Hodge star operator, $*$.  

The Hodge star operator was first used to give a decomposition of the space of $k$-forms on a Riemannian manifold.  One component of the decomposition is the space of harmonic forms.  In the Euclidean case, harmonicity defined by $*$ corresponds to the usual notion and therefore is closely related to the Cauchy-Riemann equations.  Since $\mu$ defines a Riemannian metric it is natural to consider harmonic functions in the Hodge sense and we see in Section \ref{hso} that they are closely related to solutions of \eqref{beq}.

Given an $n$-dimensional orientable Riemannian manifold $(M,g)$, let $\Omega^k(M)$ be the space of $k$-forms on $M$.  The Hodge star operator,
\begin{align*}
*\colon\Omega^k(M) \to \Omega^{n-k}(M),
\end{align*}
is a linear transformation on forms.  In addition, $*$ leaves orthonormal bases invariant when applied on a coordinate chart (see Section \ref{hso}).  We can then see that certain eigenforms of $*$ correspond to solutions of \eqref{beq}.  Furthermore, we can also define the Laplace-Beltrami operator $\Delta_g$ using $*$.  The Laplace-Beltrami operator is elliptic and we can apply the theory of elliptic PDEs to solve \eqref{beq}.  The elliptic PDE given by $\Delta_g$ is the same one as in \cite{dittmar}.  Our approach shows why the PDE is naturally linked to \eqref{beq}.

In order to prove the regularity in Theorem \ref{mainthm} one can quote elliptic regularity as in \cite{dittmar}.  We present in this paper (Section \ref{regularity}) an original proof partially based on methods in \cite[Part 4]{hellwig}.  The author of \cite{hellwig} proves a more general statement for elliptic regularity,  in our case a simplified version suffices since \eqref{beq} is a first order system.  We study an integral equation for \eqref{beq} by introducing the kernel
\begin{align*}
S(w,z) = \frac{1}{w-z + \mu(w) (\overline{w}-\overline{z})}.
\end{align*}
The solution to $\eqref{beq}$ can be written as functions integrated against $S$.  This corresponds to the holomorphic case when $\mu = 0$ and $S(w,z)$ is the Cauchy kernel.

In Section \ref{hso} we introduce basic facts about the Hodge star operator and its relation to the Beltrami equation.  Section \ref{epde} gives a solution to \eqref{beq} and proves Theorem \ref{mainthm} except for the regularity of the solution $\Phi$. In Section \ref{regularity} we prove that $\Phi \in C^1(\bb C)$ for $\mu \in C_c^\infty (\bb C)$.  This, by Section \ref{epde}, implies that $\Phi \in W^{1,2}_{\text{loc}}(\bb C)$ for $\mu$ measurable with compact support.

In Section \ref{iso} we point out the connection between Theorem \ref{mainthm} and the existence of isothermal coordinates on a Riemannian manifold of dimension two.  First we show that solving $\Delta_g$ is equivalent to finding isothermal coordinates.  Then we give a different proof for the existence of local solutions to \eqref{beq}.  This also guarantees the existence of isothermal coordinates.
The existence of a global solution in Theorem \ref{mainthm} implies the existence local solutions needed for coordinates. However, there exists a simpler proof if we assume $\mu$ is H\"older continuous.  

Using this assumption, Chern proves a local version of Theorem \ref{mainthm} in \cite{chern}.  The result guarantees $C^1$-smooth solutions.  This is weaker than Theorem \ref{mainthm} but has the advantage that it can be extended to higher dimensions.  In \cite{newlander}, the authors formulate a higher-dimensional version of \eqref{beq}.  They apply Chern's technique, not in order to get isothermal coordinates, but rather to find the conditions for the existence of a complex structure on certain even-dimensional Riemannian manifolds.  They prove that if the coefficients of the analog of $\mu$ in the higher-dimensional setting are $C^{2n}$-smooth and satisfy an integrability condition, then a complex structure exists.

We provide a proof for dimension $n = 2$ that is a simplified version of the method in \cite{chern}.  The method involves using the integral operators mentioned above.  Many technical difficulties are avoided since $\mu$ has the added regularity of H\"older continuity.

The integrability conditions for the existence of isothermal coordinates in dimensions greater than $3$ involve the Riemannian curvature tensor.  In the case of $n =3$ the condition is that the Schouten tensor vanishes and in the case of $n\ge4$ the condition is that the Weyl tensor vanishes.  Both of these tensors come from decompositions of the Riemannian curvature tensor.  When one assumes the metric is $C^3$ and the integrability condition is satisfied, then local isothermal coordinates exist.  For details regarding this see \cite{schouten}.  Much less is known when the metric has less regularity.  Similarly, the corresponding global problem is not well understood.
For more results regarding the $n > 2$ Beltrami equation see \cite{iwaniecmartin} and \cite{donaldson}.

Finally, we have attempted to make this paper entirely self-contained.  The only facts not proven are certain properties about quasiconformal maps that are shown in \cite{ahlfors}.  They are needed in Theorem \ref{mainthm}. The methods used are intended to simplify the approach to solving \eqref{beq} and to present the connection to the geometric aspects of the problem.

\subsection{Notation}
The notation used in this paper is standard.  The variable $z$ will always refer to a complex variable while $x$ and $y$ will be the standard Euclidean coordinates in $\bb R^2$.  The space $\Omega^l(\bb C)$ is the space of smooth $l$-forms on $\bb C$.  The 1-from $dz = dx + i dy$.  Integrals will always be integrated over either $dV$ or $dA$.  When $dV$ is used we refer to the volume form given a Riemannian metric.  When $dA$ is used we refer to Lebesgue measure on $\bb C$.

\subsection{Acknowledgments}
The author would like to thank Mario Bonk for introducing him to the problem and the many helpful discussions.

\section{Preliminaries}\label{hso}
In this section we review the definitions and properties of the Hodge star operator and related concepts in Riemannian Geometry.  After that we present an example of how to compute the Laplace-Beltrami operator with the Hodge star operator, given a Riemannian metric.  Lastly, we discuss the relation to the Beltrami equation.

\subsection{Hodge Star Operator}
Let $g$ be a Riemannian metric defined on $\bb C$.  The metric $g$ can be represented as a length element $ds^2$.  Classically, this can be written as
\begin{align*}
ds^2 = E dx^2 + F dxdy + G dy^2,
\end{align*}
where $E,F$ and $G \colon \bb R^2 \to \bb R$ are smooth functions.  The standard Euclidean metric is $dx^2 + dy^2$.
On the other hand, since we are in dimension two, everything can be written in complex notation.  In that case,
\begin{align*}
ds^2 = e^{\lambda}|dz+ \mu d\bar z|^2,
\end{align*}
where $\lambda \colon \bb C \to (0,\infty)$ and $\mu\colon \bb C \to \bb C$ are smooth functions and $|\mu| \le k < 1$.

Let $\Omega^l(\bb C)$ be the space of $l$-forms on $\bb C$.  For $l=0$ this coincides with smooth functions.  For $l=1$, this corresponds to elements of the form
\begin{align*}
fdz + g d\bar z,
\end{align*}
where $f$ and $g$ are smooth functions.  When $l=2$ this corresponds to elements of the form
\begin{align*}
f dz \wedge d\bar z,
\end{align*}
where $f$ is a smooth function.

Given a smooth metric $g$, we can define the Hodge star operator $*_g\colon\Omega^l(\bb C) \to \Omega^{2-l}(\bb C)$ as follows:  If $du$ and $dv$ define an orthonormal frame around $p \in \bb C$, then
\begin{align*}
*_gdu = dv,\quad*_gdv = -du
\end{align*}
for $1$-forms and
\begin{align*}
 *_g1 = du\wedge dv,\quad  *_gdu\wedge dv = 1
\end{align*}
for $0$- and $2$-forms.  Here $du \wedge dv$ is the volume form on $\bb C$.  Then, $*_g$ extends by linearity to $\Omega^k(\bb C)$ over functions.
 
The following proposition lists some properties of the Hodge star operator.
\begin{prop}\label{Hodgeprop}  Let $f,h \in C^\infty(\bb C)$ and $\omega,\eta \in \Omega^l(\bb C)$, then $*_g$ satisfies the following properties:
\begin{itemize}

\item[\emph{(1)}] $*_g(f\omega + h \eta) = f*_g\omega + h*_g\eta$

\item[\emph{(2)}] $*_g*_g\omega =(-1)^{l}\omega$

\item[\emph{(3)}] $\omega \wedge *_g\eta = \langle \omega,\bar\eta\rangle_g dV$

\end{itemize}
\end{prop}
Note that in the above proposition we use an inner product on forms.  Any Riemannian metric defines an inner product on the tangent bundle of the manifold.  This, in turn, defines an inner product on the cotangent bundle, which extends to an inner product on forms.  This inner product can then be extended to vectors over $\bb C$, which will give a Hermitian inner product.  Explicit computations of this inner product will be provided below.
\begin{proof}

Part (1) is clear by definition.

For part (2), If $\omega$ is a function, then $*_g*_g\omega = \omega*_gdu\wedge dv = \omega$.  If  $\omega$ is a $1$-form, let $\omega = fdu + h dv$.  Then $*_g*_g\omega = (f*_g*_gdu + h*_g*_gdv) = -fdu-hdv = -\omega$.  If $\omega$ is a $2$-form, let $\omega = f du \wedge dv$.  Then $*_g*_g\omega = f*_g*_gdu\wedge dv = fdu\wedge dv = \omega$.

We will only show part (3) for $1$-forms.
\begin{align*}
\omega \wedge *_g\eta &= (f_1du + h_1dv)\wedge*_g(f_2du + h_2 dv) = (f_1du + h_1dv) \wedge (f_2dv-h_2du) \\
&= (f_1f_2 + h_1h_2)du \wedge dv =\langle \omega, \bar\eta\rangle du \wedge dv,
\end{align*}
since $du$ and $dv$ form an orthonormal frame in the metric.
\end{proof}
 
Let $d \colon \Omega^l(\bb C) \to \Omega^{l+1}(\bb C)$ be the exterior derivative.  The Hodge star operator allows us to define the adjoint operator for $d$ with respect to integration on $\bb C$.
For $l \ge 1$, define $\delta\colon\Omega^l(\bb C) \to \Omega^{l-1}(\bb C)$ as 
\begin{align*}
\delta = -*_gd*_g.
\end{align*} 
Here, $\delta$ is the adjoint of $d$ only when applied to differential forms with compact support.  To see this, take the inner product
\begin{align*}
(\alpha , \beta) \coloneqq \int_{\bb C} \langle \alpha , \beta\rangle dV
\end{align*}
for $\alpha,\beta \in \Omega^l(\bb C)$.
By (3) above,
\begin{align*}
(d\omega,\eta) &= \int_{\bb C} \langle d\omega,\eta \rangle dV = \int_{\bb C} d\omega \wedge *_g\bar\eta \\
&=\int_{\bb C} d(\omega \wedge *_g\bar\eta) - (-1)^l\int_{\bb C} \omega \wedge d*_g\bar\eta
\intertext{where the first term is $0$ when $\omega$ or $\eta$ has compact support, by Stokes' theorem.  Therefore}
(d\omega,\eta) &= -\int_{\bb C} \omega \wedge *_g*_g(d*_g\eta) =  (\omega,(-*_gd*_g)\eta) = (\omega,\delta\eta).
\end{align*}

We can use $\delta$ and $d$ to define the Laplace-Beltrami operator,
\begin{align*}
\Delta_g \coloneqq d\delta + \delta d.
\end{align*}
Note that this depends on $g$ because $\delta$ depends on $g$.  When $g$ is the Euclidean metric, then $\Delta_g$ is the negative of the standard Laplacian.

\subsection{Calculating $\Delta_g$ and $*_g$ in Complex Coordinates}\label{calc}
As an example, let us first consider $\bb C$ with the Euclidean metric, $g = |dz|^2$.
By Proposition \ref{Hodgeprop}, given $\omega,\eta \in \Omega^1(\bb C)$,
\begin{align*}
\eta \wedge *_g\omega = \langle \eta,\bar\omega \rangle \frac{i}{2}dz\wedge d\bar z.
\end{align*}
As mentioned above, the inner product on forms is defined by choosing the form's vector field representative in the dual space and calculating the inner product defined by the metric.  So for example, the dual vector field of $dx$ is $\partial_x$ and the dual vector field for $dy$ is $\partial_y$.  This duality can then be extended to vectors with complex coefficients by taking a linear extension.
So the dual vector field for $dz$ is $\partial_x + i \partial_y = 2 \partial_z$.  The inner product then is defined so that it is Hermitian.  This gives that
\begin{align*}
\langle dz,d z\rangle = \langle 2\partial_z,2\partial_z\rangle = [1,-i]^T\overline{[1,-i]} = 2.
\end{align*}
So
\begin{align*}
dz\wedge *_gd\bar z &= \langle dz,d z \rangle \frac{i}{2} dz \wedge d\bar z = idz \wedge d\bar z\\
\intertext{and}
*_gd\bar z &= i d\bar z.
\end{align*}
Now, define the metric $g \coloneqq |dz + \mu d\bar z|^2$ for some $\mu \in C_c^\infty(\bb C)$, where $\|\mu\|_\infty = k < 1$.
Then the volume form becomes
\begin{align*}
dV = \frac{i}{2}(1-|\mu|^2)dz\wedge d\bar z.
\end{align*}
We can also calculate the relevant inner products.  The matrix corresponding to inner products on 1-forms is the inverse of the matrix corresponding to inner products on vector fields.  It is convenient to consider the basis $\{dz, d\bar z\}$ instead of the basis $\{dx,dy\}$.  Explicitly it is
\[A = 
-\frac{4}{(1-|\mu|^2)^2}\left( \begin{array}{cc}
\mu & -\frac{1+ |\mu|^2}{2} \\
-\frac{1+ |\mu|^2}{2} & \bar \mu \end{array} \right).
\]
So for example,
\begin{align*}
\langle dz,d z\rangle &= [1,0]^TA[0,1] = \frac{2(1+|\mu|^2)}{(1-|\mu|^2)^2}\\
\intertext{and}
 \langle dz,d\bar z\rangle &= [1,0]^TA[1,0] = -\frac{4\mu}{(1-|\mu|^2)^2}.
\end{align*}
By conjugating, the above computation also gives $\langle d\bar z, d\bar z\rangle$ and $\langle d\bar z, dz\rangle$.
Therefore,
\begin{align*}
*_gdz &= -i\frac{1+|\mu|^2}{1-|\mu|^2}dz - 2i\frac{\mu}{1-|\mu|^2}d\bar z\\
\intertext{and}
*_g d \bar z &= 2i\frac{\bar \mu}{1-|\mu|^2}dz + i\frac{1 + |\mu|^2}{1-|\mu|^2}d\bar z.
\end{align*}
To simplify notation, let 
\begin{align}
a = \frac{1+|\mu|^2}{1-|\mu|^2} \label{a}
\end{align}
and 
\begin{align}
b = \frac{2\mu}{1-|\mu|^2}. \label{b}
\end{align}
Let $f: \bb C \to \bb C$ be a $C^\infty$-smooth function.  We can now compute $\Delta_g f$.
\begin{align}
*_gdf &= ( -ia\partial_zf + i\bar b \partial_{\bar z}f)dz +(-ib \partial_zf + ia\partial_{\bar z}f)d\bar z. \label{dmu1}\\
\intertext{Next,}
d*_g(df) &=  (\partial_{\bar z}( ia\partial_zf - i\bar b \partial_{\bar z}f) +\partial_z(-ib \partial_zf + ia\partial_{\bar z}f))dz \wedge d\bar z.\nonumber\\
\intertext{Finally,}
*_g(d*_g(df)) &= \frac{2}{i(1-|\mu|^2)}(\partial_{\bar z}( ia\partial_zf - i\bar b \partial_{\bar z}f) +\partial_z(-ib \partial_zf + ia\partial_{\bar z}f)).\nonumber
\end{align}

\subsection{Beltrami Equation}\label{be}
For $\mu$ as above, the Beltrami equation, \eqref{beq}, is
\begin{align*}
\partial_{\bar z} f = \mu\partial_z f.
\end{align*}
When $\mu \equiv 0$ the above corresponds to the Cauchy-Riemann equations.  Let 
\begin{align*}
d_\mu \coloneqq \frac{1}{2}(1-i*_\mu)d,
\end{align*}
where $*_\mu$ is the Hodge star operator for the metric $g = |dz + \mu d\bar z|^2$.  We write here $*_\mu$ instead of $*_g$ to emphasize that the Hodge star operator depends only on $\mu$.

In this Hodge star formalism the Cauchy-Riemann equations can be written as
\begin{align*}
d_0f = 0.
\end{align*}
For a general $\mu$, one gets a similar result.
\begin{lemma}\label{formbeq}
A function $f\in C^\infty(\bb C)$ satisfies $d_\mu f = 0$ if and only if $f$ satisfies \eqref{beq}.
\end{lemma}
\begin{proof}
We wish to calculate $d_\mu f$.
In Section  \ref{calc}, the term $*_\mu df$ was already calculated.  We see from \eqref{dmu1} that
\begin{align*}
*_\mu df = ( -ia\partial_zf + i\bar b \partial_{\bar z}f)dz +(-ib \partial_zf + ia\partial_{\bar z}f)d\bar z.
\end{align*}
So
\begin{align*}
d_\mu f = \frac{1}{2}(-a\partial_zf+\bar b \partial_{\bar z} f + \partial_zf)dz +\frac{1}{2}(-b\partial_zf+a\partial_{\bar z}f +\partial_{\bar z}f)d\bar z.
\end{align*}
Setting this to $0$ and looking at the $d\bar z$ component gives,
\begin{align*}
\partial_{\bar z}f + a\partial_{\bar z}f= b\partial_zf.
\end{align*}
By \eqref{a} and \eqref{b},
\begin{align*}
(1-|\mu|^2 +1 +|\mu|^2) \partial_{\bar z}f = 2\mu\partial_z f.
\end{align*}
This is the Beltrami equation.  
The $dz$ component gives
\begin{align*}
\partial_z f - a\partial_z f = - \bar b\partial_{\bar z} f.
\end{align*}
We again use \eqref{a} and \eqref{b} to get that
\begin{align*}
(1-|\mu|^2 - (1+|\mu|^2))\partial_z f = -2\bar \mu\partial_{\bar z} f.
\end{align*}
If $\mu = 0$ this equation is trivial.  Otherwise, it simplifies to \eqref{beq}.
\end{proof}
We now continue this type of calculation to get a connection between harmonic functions with respect to $\mu$ and solutions to \eqref{beq}.
\begin{lemma}\label{formharm}
Suppose $f \in C^\infty(\mathbb C)$ satisfies 
\begin{align*}
\Delta_\mu f = 0.
\end{align*}
Then there exists a solution to \eqref{beq} of the form $g-if$ where $dg = *_\mu df$.
\end{lemma}
\begin{proof}
Recall that
\begin{align*}
\Delta_\mu f = -*_\mu d(*_\mu df)
\end{align*}
So $*_\mu df$ is closed and there exists $g$ such that $dg = *_\mu df$.  By Proposition \ref{Hodgeprop},
\begin{align*}
*_\mu dg &= *_\mu (*_\mu df) = -df\\
\intertext{and}
-i*_\mu df &= -idg.\\
\intertext{Adding these two equations gives}
*_\mu d(g-if) &= -id(g-if),\\
\intertext{which is}
d_\mu(g-if) &= 0.
\end{align*}
By Lemma \ref{formbeq}, $g-if$ is the desired solution to the Beltrami equation.
\end{proof}

\section{Elliptic PDE Methods}\label{epde}

The goal of this section is to prove the existence and uniqueness of solutions to the Beltrami equation
\begin{align*}
\partial_{\bar z} f = \mu \partial_z f
\end{align*}
for $\mu \colon\bb C \to \bb C$ with $|\mu(z)| \le k < 1$.
We have seen above that the Beltrami equation can be expressed in terms of the Hodge star as
\begin{align*}
d_\mu f \coloneqq \frac{1}{2}(1-i*_\mu)df = 0.
\end{align*}
By Lemma \ref{formharm} the existence of a solution $f$ is equivalent to the existence of a solution of Laplace's equation, $\Delta_\mu f = 0$.  It is known that locally solutions always exist (this will be proven in Section \ref{iso}, for the standard proof see \cite[Ch. 5.11]{taylor}).  To find a global solution on $\bb C$ more work is required.
We will prove the following theorem.
\begin{theorem}\label{mainthm2}
Let $\mu \in C_c^\infty(\bb C)$ with $\|\mu\|_\infty = k < 1$.  There exists an orientation-preserving homeomorphism $\Phi \colon \bb C \to \bb C$ that solves the Beltrami equation,
\begin{align*}
\partial_{\bar z} \Phi = \mu \partial_z \Phi,
\end{align*}
with $\Phi \in C^1(\bb C)$.
\end{theorem}

\subsection{Suitable Hilbert Space}

Define $\cc H$ as the space of $u \in L^1_{\text{loc}}(\bb C)$ such that $|du| \in L^2(\bb C)$ with the equivalence relation that $u \sim v$ if $u-v$ is constant a.e. Note that in the following $dA$ refers to Lebesgue measure on $\bb C$.

\begin{prop}
 $\cc H$ is a Hilbert space with the inner product
\begin{align*}
([u],[v]) = \int_{\bb C} \langle du, dv\rangle  dA,
\end{align*}
where $u$ and $v$ are representatives of $[u]$ and $[v]$ respectively.
\end{prop}
\begin{proof}
If $u \sim v$ and $\phi \in \Omega^1(\bb C)$ with compact support, then
\begin{align*}
\int (du - dv)\wedge \phi = -\int (u -v) d\phi= -c\int d\phi = 0,
\end{align*}
since $\phi$ has compact support. 
Thus $du = dv$ a.e. and $([u],[v])$ is well-defined.  Linearity and symmetry up to conjugation both hold.
It is clear that $([u],[u]) \ge 0$.
Also, $([u],[u]) = 0$ if and only if $du = 0$, which holds if and only if $u$ is constant.  So $u \sim 0$.  Therefore we have positive definiteness.

$\cc H$ is complete:  Suppose $([u_k])$ is a Cauchy sequence.  By our definition, $|du_k| \in L^2(\bb C)$, so $du_k \to v$ where $|v| \in L^2(\bb C)$.  There exists a representative of $[u_k]$ such that 
\begin{align*}
\frac{1}{\pi}\int_{B(0,1)} u_k dA = 0.
\end{align*}
The functions $u_k$ are in $L^1_{\text{loc}}(\bb C)$.  So by the Poincar\'e inequality, 
\begin{align*}
\|u_m - u_n\|_{L^1(B(0,r))} \le C_r\|du_m - du_n\|_{L^2(B(0,r))}.
\end{align*}
Therefore, for a fixed $r$, there are representatives $u_k$ such that $u_k \to u_r$ in $L^1(B(0,r))$.  Since the $u_k$ do not depend on $r$, the $u_r$ must agree with each other a.e.\ where they are defined.  Therefore $u_k \to u$ in $L^1_{\text{loc}}(\bb C)$.  

For convergence we now show that $du = v$.  If $\phi$ is a smooth, compactly supported 1-form, then
\begin{align*}
\int ud\phi = \lim_{n \to \infty} \int u_n d\phi = - \lim_{n \to \infty} \int du_n \wedge \phi dA = - \int v\wedge\phi dA.
\end{align*}
So $du = v$ a.e.
This makes $\cc H$ a Hilbert space. 
\end{proof}

\begin{lemma}
Define
\begin{align*}
([v],[u])_\mu = \int_{\bb C}dv\wedge *_\mu d\bar u.
\end{align*}
for $[u],[v] \in \cc H$.  This is an inner product on $\cc H$ that makes $\cc H$ a Hilbert space.
\end{lemma}
\begin{proof}
The inner product $(\cdot,\cdot)_\mu$ is well-defined, since $|du|,|dv| \in L^2(\bb C)$.
Also, by Proposition \ref{Hodgeprop} (3),
\begin{align*}
 \int_{\bb C} dv\wedge *_\mu d\bar u = \int_{\bb C} \langle dv,du\rangle_\mu dV.
\end{align*}
Since the inner product on forms is Hermitian, integrating must give a Hermitian inner product as well.  To show completeness note that
\begin{align*}
([u],[u])_\mu &= \int_{\bb C} a\bigg(\bigg |\frac{\partial u}{\partial z}\bigg |^2 + \bigg |\frac{\partial u}{\partial \bar z}\bigg |^2\bigg ) - 2\text{Re}\bigg (b\frac{\partial u}{\partial z} \frac{\partial \bar u}{\partial z}\bigg )dA,\\
\intertext{where $a$ and $b$ are from \eqref{a},\eqref{b} (see Section \ref{calc}).}
\text{Re}\bigg (b\frac{\partial u}{\partial z}\frac{\partial \bar u}{\partial z}\bigg ) &\le|b|\bigg | \frac{\partial u}{\partial z}\frac{\partial \bar u}{\partial z}\bigg | \le \frac{|b|}{2}\bigg(\bigg |\frac{\partial u}{\partial z}\bigg |^2 + \bigg |\frac{\partial u}{\partial \bar z}\bigg |^2\bigg ) . \\
\intertext{So}
([u],[u])_\mu &\ge  \int_{\bb C}a\bigg(\bigg |\frac{\partial u}{\partial z}\bigg |^2 + \bigg |\frac{\partial u}{\partial \bar z}\bigg |^2\bigg ) - |b|\bigg(\bigg |\frac{\partial u}{\partial z}\bigg |^2 + \bigg |\frac{\partial u}{\partial \bar z}\bigg |^2\bigg ) dA.\\
\intertext{and}
a - |b| &= \frac{1+|\mu|^2 - 2|\mu|}{1-|\mu|^2} = \frac{(1-|\mu|)^2}{1- |\mu|^2} = \frac{1-|\mu|}{1+|\mu|}.\\
\intertext{So}
([u],[u])_\mu &\ge \frac{1-k}{1+k}\|u\|_{\cc H}^2.
\end{align*}
Completeness with the new inner product then follows from $\cc H$ being complete. 
\end{proof}

\subsection{Inhomogeneous Beltrami Equation}

Finding a solution for $id*_\mu (du) = 0$ may be difficult.  So instead, we first try to solve an inhomogeneous version.  Suppose $F,G \in C_c^\infty (\bb C)$, and let $\eta = G dz - F d\bar z$.
\begin{prop}
There exists $f \in \cc H\cap C(\bb C)$ such that 
\begin{align}
id(*_\mu df) = d\eta
\end{align}
in distribution.
\end{prop}
\begin{proof}
Define, for $[v] \in \cc H$,
\begin{align*}
L([v]) \coloneqq  i\int_{\bb C} v\wedge d\eta.
\end{align*}
Let $c \in \bb C$, if $v$ and $v+c$ are two representatives of $[v]$, then
\begin{align*}
L(v) - L(v+c) = i\int_{\bb C} cd\eta = 0,
\end{align*}
since $\eta$ has compact support.  So $L$ is well-defined.
We need to show that $L$ is a continuous linear functional in the norm defined by $\mu$. Suppose $v \in C_c^\infty(\bb C)$, then
\begin{align*}
|L(v)|^2 = \bigg |  \int_{\bb C} vd\eta\bigg |^2 = \bigg |  \int_{\bb C} dv\wedge \eta\bigg |^2 &\le \|v\|_{\cc H}^2 \int_{\bb C} |F|^2 + |G|^2 dA\\
&\le K\|v\|_{\mu}^2 \int_{\bb C} |F|^2 + |G|^2 dA
\end{align*}
by the lemma.
The functional $L$ is continuous on a dense subset of $\cc H$ and by taking limits it is continuous on $\cc H$.
By the Riesz representation theorem there exists a function $f \in \cc H$ such that $L(v) = (v,\bar f)_\mu$ for all $v \in \cc H$.  So for all $v \in C_c^\infty(\bb C)$,
\begin{align*}
\int_{\bb C} dv\wedge (*_\mu df) = i\int_{\bb C}v\wedge d\eta.
\end{align*}
Applying integration by parts to the left hand side, we get that
\begin{align*}
\int_{\bb C} v\wedge (d*_\mu df) = -i\int_{\bb C}v\wedge d\eta.
\end{align*}
So in distribution, $id(*_\mu df) = d\eta$ and 
$f$ solves the inhomogeneous Laplace equation. 
By elliptic regularity, if $\mu \in C^\infty (\bb C)$, then $f \in C(\bb C)$ (see \cite[Ch. 6]{trudinger}).  A proof for the continuity of $f$ is also given in Section \ref{regularity} below.
\end{proof}
We would like to apply the Poincar\'e lemma to get a solution to the Beltrami equation from $f$.  To do this, we need to prove a version of the Poincar\'e lemma for $L^2$-functions.
\begin{lemma}\label{poin}
Suppose $u,v \in L^2(\bb C)$ and $\partial_{\bar z}u = \partial_z v$ in distribution.  Then there exists a function $g \in L^2_{\text{loc}}(\bb C)$ such that $\partial_z g = u$ and $\partial_{\bar z} g = v$.
\end{lemma}
\begin{proof}
Let $\varphi_\epsilon$ be an approximation to the identity (i.e.\ $\varphi_1$ has compact support, $\int_{\bb C}\varphi_1 dA = 1$, and $\varphi_\epsilon = \epsilon^{-n}\varphi(\cdot/\epsilon) \to \delta$ in distribution as $\epsilon \to 0$, where $\delta$ is the Dirac measure).  Then $u_\epsilon : = u*\varphi_\epsilon$ and $v_\epsilon \coloneqq  v*\varphi_\epsilon \in C^\infty$ converge to $u$ and $v$ respectively in $L^2(\bb C)$.

By the Poincar\'e lemma in the smooth case, there exists a function $g_\epsilon$ such that $\partial_zg_\epsilon = u_\epsilon$ and $\partial_{\bar z} g_\epsilon = v_\epsilon$.  We can normalize the $g_\epsilon$ so that
\begin{align*}
\int_{B(0,1)} g_\epsilon dA = 0.
\end{align*}
By the Poincar\'e inequality, $g_\epsilon \in L^2_{\text{loc}}(\bb C)$ with bounded $L^2$-norm on any fixed ball.  The Banach-Alaoglu theorem implies that there exists a subsequence $(g_{\epsilon_n})$ that converges to $g$ in $L^2_{\text{loc}}(\bb C)$.  It then can be seen that $\partial_z g = u$ and $\partial_{\bar z} g = b$ by integrating against test functions and passing to the limit.
\end{proof}
The proposition and the lemma give the following corollary,
\begin{corollary}
There exists $\Psi \in C(\bb C)$ such that 
\begin{align*}
d_\mu\Psi = -\frac{1}{2}(1 - i*_\mu)\eta.
\end{align*}
\end{corollary}
\begin{proof}
We would like to construct a function that solves the inhomogeneous Beltrami equation.
Define
\begin{align*}
\omega \coloneqq i*_\mu (df) - \eta.
\end{align*}
By this definition, $\omega$ is closed and therefore exact on $\bb C$. Thus, there exists a function $g$, by Lemma $\ref{poin}$, such that 
\begin{align}
dg = \omega &= i*_\mu df-\eta.\nonumber\\ 
\intertext{Now we can do a similar calculation as in Section \ref{be}.}
*_\mu (dg) &= i*_\mu *_\mu (df) -*_\mu \eta = -i(df)-*_\mu \eta\nonumber\\
\intertext{and}
*_\mu (df) &=-idg -i\eta.\nonumber\\
\intertext{Adding these equations gives,}
*_\mu (df+dg) &= -i(df+dg) -(*_\mu +i)\eta,\nonumber\\
\intertext{which rearranges to}
d_\mu(f+g) &= -\frac{1}{2}(1 - i*_\mu)\eta. \label{output}
\end{align}
\end{proof}
If $\eta = (\partial_z \mu )d\bar z$ and $\Psi = f+g$, then \eqref{output} can be rewritten as
\begin{align*}
\partial_{\bar z} \Psi-  \mu\partial_z \Psi  =\partial_z \mu.
\end{align*}
To prove the theorem above, we need a solution to the homogeneous Beltrami equation.  Suppose $\alpha = udz + vd\bar z$ is closed and $\mu u = v$.  Then
\begin{align*}
\partial_{\bar z} u &=\partial_z v 
= \partial_z (\mu u)\\ &= (\partial_z \mu) u + \mu(\partial_z u).
\end{align*}
Dividing by $u$, gives
\begin{align*}
\partial_{\bar z} (\log u) = \partial_z \mu + \mu\partial_z (\log u).
\end{align*}
This is the inhomogeneous Beltrami equation.  Define $\alpha = e^\Psi d z + \mu e^\Psi d \bar z$, then
$\alpha$ is closed.  The Poincar\'e lemma still applies if $e^\Psi$ and $\mu e^\Psi \in C(\bb C) \cap W^{1,1}_{\text loc}(\bb C)$.
So there exists a function $\Phi \in C^1(\bb C)$ with $d\Phi = \alpha$.  This gives that $\partial_{\bar z} \Phi = \mu \partial_z \Phi$.
We also have that $\partial_z \Phi = e^\Psi \ne 0$. So $\Phi$ is locally injective.  
Since $\mu$ has compact support, $\Psi$ and $\Phi$ are holomorphic for large $z$.  The power series for $\Psi$ near $\infty$ is
\begin{align*}
\sum_{n = 0}^\infty a_nz^{-n} + \sum_{n=1}^\infty b_n z^n.
\end{align*}
Also, we have that $\partial_z \Psi = 2\partial_z f$ when $\mu = 0$.  The function $f$ is in $\cc H$ so $\partial_z \Psi \in L^2(\bb C)$. Therefore, $b_n = 0$ for all $n \in \bb N$ and
\begin{align*}
\Psi(z) = \sum_{n=0}^\infty a_nz^{-n}.
\end{align*}
So
\begin{align*}
\lim_{z \to \infty} |\Phi'(z)| = \lim_{z \to \infty} |e^{\Psi(z)}| = |e^{a_0}| \ne 0.
\end{align*}
 This means that $\lim_{z \to \infty} |\Phi(z)| = \infty$ and that $\Phi$ is a locally injective and proper map on $\bb C$.  This is enough to say that $\Phi$ is a homeomorphism.  $\Phi$ is orientation-preserving because the Jacobian of $\Phi$, $J_\Phi = |\partial_z \Phi|^2(1- |\mu|^2) > 0$.  This proves Theorem \ref{mainthm2}.  We now proceed to the main result.

\begin{proof}[Proof of Theorem \ref{mainthm}]
Recall that $\mu \in L^\infty(\mathbb C)$ and $\operatorname{supp}(\mu) \subset B(0,R)$.  Let $(\mu_n) \in C_c^\infty(\bb C)$ such that $\operatorname{supp}(\mu_n)\subset B(0,R)$ for all $n \in \bb N$, $\|\mu_n\|_\infty \le \|\mu\|_\infty$ and $\mu_n \to \mu$ a.e.  By Theorem \ref{mainthm2}, there exists a sequence $(\Phi^n)$ such that $\Phi^n$ solves the Beltrami equation for $\mu_n$.
The $\Phi^n$ are $K$-quasiconformal for a fixed $K$ and therefore form a normal family (see \cite{ahlfors} or \cite{aim}).  So there exists a subsequence, which we will still call $(\Phi^n)$, that converges locally uniformly to $\Phi$.  The limit of $K$-quasiconformal maps is $K$-quasiconformal or constant.  Near $\infty$ the $\Phi^n$ are all holomorphic and converge locally uniformly.  So $\Phi$ is holomorphic near $\infty$.  If we normalize as above to have $\partial_z \Phi^n(z) - 1 \in L^2(\bb C)$, then
\begin{align*}
\Phi^n(z) = z + b_n + \cc O(\frac{1}{z})
\end{align*}
near $\infty$.  Therefore,
\begin{align*}
\Phi(z) = z + b + \cc O(\frac{1}{z})
\end{align*}
and $\partial_z \Phi - 1 \in L^2(\bb C)$.  This also shows that $\Phi$ is not constant and must be a $K$-quasiconformal map.

Let $\varphi \in C_c^\infty(\bb C)$,
\begin{align*}
\int \varphi_z \Phi dA = \lim_{n \to \infty} \int \varphi_z \Phi^n dA= -\lim_{n \to \infty} \int \varphi \Phi^n_zdA.
\end{align*}
And
\begin{align*}
\int \varphi_{\bar z} \Phi dA= \lim_{n \to \infty} \int \varphi_{\bar z} \Phi^n dA= -\lim_{n \to \infty} \int \varphi \Phi^n_{\bar z}dA.
\end{align*}
So $\Phi_z^n \to \Phi_z$ and $\Phi_{\bar z}^n \to \Phi_{\bar z}$ weakly. 
The weak derivatives of $\Phi$ are in $L^2_{\text{loc}}(\bb C)$ and $\Phi \in W^{1,2}_{\text{loc}}(\bb C)$.

Also $\Phi$ solves the correct Beltrami equation:
\begin{align*}
\int \varphi \mu^n \Phi_z^n - \varphi \mu \Phi_z = \int \varphi(\mu_n - \mu)\Phi_z^n + \int \varphi\mu(\Phi_z^n - \Phi_z).
\end{align*}
The first integral can be bounded:
\begin{align*}
\bigg | \int \varphi(\mu_n - \mu)\Phi_z^n \bigg | \le \|\sqrt{|\varphi|}\Phi_z^n\|_2\|\sqrt{|\varphi|}(\mu_n-\mu)\|_2 \to 0,
\end{align*}
since $\mu_n \to \mu$ pointwise and $\Phi_z^n$ is bounded in $L^2(\bb C)$.
The second integral converges to $0$ as well, because $\mu$ is bounded and $\Phi_z^n \to \Phi_z$ weakly.
This means that $\mu^n \Phi_z^n \to \Phi_{\bar z}$ weakly and $\mu^n \Phi_z^n \to \mu \Phi_z$ weakly as $n \to \infty$. So $\Phi$ solves the Beltrami equation for $\mu$.

All that is left to show is that our solution $\Phi$ is unique when $\Phi(0) = 0$, and $\Phi(1) = 1$.
Let $\widetilde \Phi$ satisfy \eqref{beq} and have the above normalization.  Then $\widetilde \Phi \circ \Phi^{-1}$ is a homeomorphism. It is necessary to show that $\widetilde \Phi \circ \Phi^{-1}$ is quasiconformal.  This is true since the composition of quasiconformal maps is still quasiconformal.  Finally, we need that the Beltrami coefficient is $0$.  Quasiconformal maps satisfy the chain rule so this is true (see \cite[Ch. 2]{ahlfors}).

So $\widetilde \Phi \circ \Phi^{-1}$ is conformal from $\bb C$ to itself.  The above normalizations make $\widetilde \Phi \circ \Phi^{-1} (z)  = z$ and $\widetilde \Phi = \Phi$.  This finishes the  proof of Theorem \ref{mainthm}.
\end{proof}

\section{Regularity}\label{regularity}
In this section we show that our solution to \eqref{beq} is H\"older continuous.  As mentioned above this follows from elliptic regularity.  We have decided to include a different proof so that the paper is self-contained and since the reader may find the method interesting.  The proof is similar in nature to the classical proof for elliptic regularity but has been simplified to fit our problem.

\subsection{Preliminary Statements}
\begin{lemma}\label{kernelcont}
Let $u \in L^p_{\text{loc}}(\bb C)$ for $p > 2$ and let $K \colon \bb C \times \bb C \to \bb C$ satisfy
\begin{align}
|K(w,z)| &\le C\frac{1}{|w-z|} \label{kboundi}\\
\intertext{and}
|K(w_1,z) - K(w_2,z)| &\le C'|w_1-w_2|\bigg (\frac{1}{|w_1-z|^2} + \frac{1}{|w_2-z|^2} \bigg ). \label{kboundii}\\
\end{align}
Also assume $\operatorname{supp}(K(w,\cdot))$ is compact.
Then 
\begin{align*}
v(w) = \int_{\bb C}u(z)K(w,z)dz
\end{align*}
is H\"older-continuous.
\end{lemma}
\begin{proof}
\begin{align*}
|v(w_1) - v(w_2)| &\le \int_{\bb C} |u(z)||K(w_1,z) - K(w_2,z)| dz \\
&= \int_{B(w_1,R)} |u(z)||K(w_1,z) - K(w_2,z)| dz +  \int_{\bb C\setminus B(w_1,R)} |u(z)||K(w_1,z) - K(w_2,z)| dz,
\end{align*}
where $R = 2|w_1 - w_2|$.  Then, by H\"older's inequality, the first term,
\begin{align*}
\int_{B(w_1,R)} |u(z)||K(w_1,z) - K(w_2,z)| dz \le \|u\mathbbm{1}_{\text{supp(K)}}\|_p\|K(w_1,z) - K(w_2,z)\|_q,
\end{align*}
where $1/p+1/q = 1$.  By \eqref{kboundi},
\begin{align*}
\bigg (\int_{B(w_1,R)} |K(w_1,z)&-K(w_2,z)|^q dz \bigg )^{1/q} \\ 
&\le C\bigg (\int_{B(w_1,R)}\frac{1}{|w_1-z|^q} dz\bigg )^{1/q}+C\bigg (\int_{B(w_1,R)} \frac{1}{|w_2-z|^q} dz\bigg )^{1/q}.\\
\intertext{Since $w_1,w_2 \in B(w_1,R)$, we get that}
&\lesssim CR^{(2-q)/q}.
\end{align*}
Analyzing the second term, by \eqref{kboundii}, we get that
\begin{align*}
\int_{\bb C\setminus B(w_1,R)} |u(z)||K(w_1,z) &- K(w_2,z)| dz\\& \le \frac{C'R}{2}\int_{\bb C\setminus B(w_1,R)} |u(z)|\bigg (\frac{1}{|w_1-z|^2} + \frac{1}{|w_2-z|^2} \bigg )dz.
\end{align*}
Then, by H\"older's inequality, this is bounded above by
\begin{align*}
\frac{C'R}{2}\|u\mathbbm{1}_{\operatorname{supp}(K)}\|_p\bigg (\int_{\bb C \setminus B(w_1,R)}\bigg (\frac{1}{|w_1-z|^2} + \frac{1}{|w_2-z|^2} \bigg )^q dz\bigg )^{1/q}.
\end{align*}
Note however, that
\begin{align*}
d(w_2,\partial B(w_1,R)) \ge |w_1-w_2| \ge |w_1 - z|
\end{align*}
for $z \notin B(w_1,R)$.  So
\begin{align*}
\int_{\bb C \setminus B(w_1,R)}\bigg (\frac{1}{|w_1-z|^2} + \frac{1}{|w_2-z|^2} \bigg )^qdz &\le \int_{\bb C \setminus B(w_1,R)} \frac{1}{|w_1-z|^{2q}} dz\\
&=\frac{R^{2-2q}}{2-2q}.
\end{align*}
Putting all of these bounds together we get that
\begin{align*}
|v(w_1) - v(w_2)| \lesssim \|u\mathbbm{1}_{\operatorname{supp}(K)}\|_p(R^{(2-q)/q} + R^{(2-q)/q}) = 2\|u\mathbbm{1}_{\operatorname{supp}(K)}\|_p|w_1 - w_2|^{1-2/p},
\end{align*}
where the constants depend only on $C,C'$.  Since $p > 2$ we see that $v$ is H\"older-continuous.
\end{proof}
\begin{lemma}\label{riesz}
If $u \in L^p_{\text{loc}}(\bb C)$, then for fixed $R > 0$,
\begin{align*}
I_1(|u|)(w) = \int_{B(0,R)}\frac{|u(z)|}{|w-z|}dz
\end{align*}
is in $L^{2p/(2-p)}$.
\end{lemma}
\begin{sproof}
Divide $I_1(|u|)(w)$ into an integral on a ball around $w$ and an integral on the complement of a ball around $w$.  The integral bounded away from $w$ is in $L^{2p/(2-p)}(\bb C)$ by H\"older's inequality.  The integral around $w$ can be bounded by the maximal function of $u$.  For details see \cite[ch. 3]{heinonen}.
\end{sproof}

\subsection{Pseudo-Fundamental Solution}

We will use the following notation in this section: 
\begin{align*}
\partial_{\bar z}^\mu f &= \partial_{\bar z}f - \mu(z) \partial_z f,\\
\intertext{and}
\partial_{\bar z}^{\mu *} &= \partial_{\bar z}f - \partial_z(\mu(z) f).
\end{align*}
Note that $\partial_{\bar z}^{\mu *}$ is the adjoint of $\partial_{\bar z}^\mu f $.

Recall that $\mu$ is a $C^\infty$-smooth function with compact support defined on $\bb C$.  We now define a pseudo-fundamental solution for the Beltrami equation.  Let
\begin{align}
S(w,z) = \frac{1}{\pi}\frac{1}{w-z + \mu(w)(\bar w - \bar z)}.
\end{align}
The denominator satisfies,
\begin{align*}
|w+\mu(w)\bar w-(z+\mu(w) \bar z)| \ge |w-z|  - k |\bar w - \bar z| = (1-k)|w-z|. 
\end{align*}
So
\begin{align}
|S(w,z)| \le \frac{1}{\pi(1-k)}\frac{1}{|w-z|}. \label{boundi}
\end{align}
When $w\ne z$, we have that $S(w,z)$ is smooth and that
\begin{align}
|\partial_{\bar z}^{\mu*} S(w,z)| &\le C_1\frac{|\mu(w) - \mu(z)|}{|w-z|^2} + C_2\frac{|\partial_z\mu(z)|}{|w-z|}.\nonumber\\
\intertext{So}
|\partial_{\bar z}^{\mu*} S(w,z)| &\le \frac{C'}{|w-z|}, \label{boundii}
\end{align}
where $C_1,C_2,$ and $C'$ depend only on $\mu$.  We constructed $S(w,z)$ so that
\begin{align}
\partial_{\bar z}S(w,z) - \mu(w)\partial_zS(w,z) = 0. \label{frozcoef}
\end{align}
This means that $S(w,z)$ is a fundamental solution for \eqref{beq} when $\mu$ is set to be the constant $\mu(w)$.

A ``Green's theorem"-type statement can be proved for $S(w,z)$.  Let $D' \subset D$ be disks centered at some point $w_0 \in \bb C$.
\begin{lemma}\label{greenlemma}
For $\varphi \in C_c^\infty(D')$,
\begin{align}
\int_{D'}\varphi(w)\partial_{\bar z}^{\mu*} S(w,z)dw = \varphi(z) + \partial_{\bar z}^{\mu*}\bigg (\int_{D'} \varphi(w)S(w,z)dw\bigg ). \label{greentype}
\end{align}
\end{lemma}
By  \eqref{boundii} the integral on the left hand side is well-defined.  By \eqref{boundi}, the integral on the right hand side is well-defined and $C^1$-smooth. 
\begin{proof}
In order to prove \eqref{greentype} first we will show that
\begin{align}
\int_D \varphi(z)\partial_{\bar z}^{\mu*} S(w,z) dz = \varphi(w) - \int_D S(w,z)\partial_{\bar z}^\mu\varphi(z) dz. \label{greentypei}
\end{align}
Subtracting \eqref{frozcoef}, the left hand side becomes
\begin{align*}
\int_D \varphi(z)((\mu(w) - \mu(z))S(w,z))_z dz.
\end{align*}
We can use integration by parts by \eqref{boundi}.  So the above term becomes
\begin{align*}
-\int_D\varphi(z)_z(\mu(w)&-\mu(z))S(w,z)dz \\
&=  -\int_D S(w,z)(\varphi(z)_{\bar z} - \mu(z)\varphi(z)_z) dz + \int_D S(w,z)( \varphi(z)_{\bar z} -\mu(w)\varphi(z)_z) dz.
\end{align*}
We would like the second term to be $\varphi(w)$.  To see this, consider the change of variable 
\begin{align*}
\zeta = z+ \mu(w)\bar z, \xi = w + \mu(w) \bar w.
\end{align*}
Then, the second integral becomes
\begin{align*}
\int_D S(w,z)( \varphi(z)_{\bar z} -\mu(w)\varphi(z)_z) dz = \frac{1}{\pi}\int_D \frac{1}{\xi - \zeta} \varphi(z(\zeta))_{\bar \zeta} d\zeta = \varphi(w),
\end{align*}
by the Cauchy-Green formula.
This shows \eqref{greentypei}.  To show \eqref{greentype} we integrate the left hand side of \eqref{greentype} against another test function $\psi$ and apply Fubini's theorem:
\begin{align*}
\int_D\psi(z)\int_{D'}\varphi(w)\partial_{\bar z}^{\mu*} S(w,z) dw dz &= \int_{D'}\varphi(w)\int_D\psi(z)\partial_{\bar z}^{\mu*} S(w,z) dz dw\\
&=\int_{D'} \varphi(w) \psi(w)dw - \int_{D'}\varphi(w) \int_D S(w,z)\partial_{\bar z}^\mu \psi(z) dzdw\\
&= \int_{D'} \varphi(w) \psi(w)dw  - \int_{D}\partial_{\bar z}^\mu \psi(z) \int_{D'} S(w,z)\varphi(w)dw dz\\
&=\int_D \varphi(z) \psi(z)dz + \int_D \psi(z) \partial_{\bar z}^{\mu*}\bigg (\int_{D'} S(w,z) \varphi(w) dw\bigg ) dz
\end{align*}
by \eqref{greentypei}, Fubini's theorem and integration by parts.
This is true for an arbitrary test function $\psi$, so \eqref{greentype} is true a.e.
\end{proof}
\subsection{Continuity of the Solution}
Let $D',D$ be disks centered at $w_0$ with $\overline {D'} \subset D$.  For $F \in C_c^\infty(D)$, define
\begin{align*}
\tilde{g}(w) \coloneqq \int_D g(z) \partial_{\bar z}^{\mu*}(\rho(z)S(w,z)) dz + \int_D F(z) \rho(z)S(w,z)dz,
\end{align*}
where $\rho \in C_c(D)$ and $\rho(z) = 1$ for $z \in D'$.
\begin{lemma}
Let $g \in L^1_{\text{loc}}(\bb C)$ and $|\nabla g| \in L^2(\bb C)$ with
\begin{align*}
\partial_{\bar z}^\mu g = F.
\end{align*}
in distribution.
Suppose $\mu, F \in C_c^\infty(\bb C)$ and $|\mu|\le k < 1$.
Then $g(w) = \tilde{g}(w)$ for a.e.\ $w \in D'$.
\end{lemma}
\begin{proof}
Let $\varphi \in C_c^\infty(D')$, then
\begin{align}
\int_{D'} \varphi(w)\tilde{g}(w)dw &= \int_D g(z) \int_{D'}\varphi(w) \partial_{\bar z}^{\mu*}(\rho(z)S(w,z))dw dz  \nonumber \\
&\text{\ \ \ \ \ \ \ \ \ \ \ \ \ \ \ \ \ \ \ \ \ \ \ \ }+ \int_D F(z)\int_{D'} \varphi(w) \rho(z)S(w,z)dw dz\label{visu1} \\  
&= I + II. \nonumber
\end{align}
The use of Fubini's theorem is justified here by equations \eqref{boundi} and \eqref{boundii}.

Then
\begin{align*}
I &= \int_{D'} g(z) \int_{D'}\varphi(w) \partial_{\bar z}^{\mu*} S(w,z)dw dz + \int_{D\setminus D'} g(z) \int_{D'}\varphi(w) \partial_{\bar z}^{\mu*}(\rho(z)S(w,z))dw dz \\
&= \int_{D'} g(z)\varphi(z) dz + \int_{D'}g(z)\partial_{\bar z}^{\mu*}\bigg(\int_{D'} \varphi(w)S(w,z) dw\bigg) dz\\
& \text{\ \ \ \ \ \ \ \ \ \ \ \ \ \ \ \ \ \ \ \ \ \ \ \ }+  \int_{D\setminus D'} g(z) \int_{D'}\varphi(w) \partial_{\bar z}^{\mu*}(\rho(z)S(w,z))dw dz
\end{align*}
by equation \eqref{greentype}.
In the last integral $z \in D\setminus D'$ and therefore bounded away from $w$.  We can take the derivative outside of the integral and combine the last two terms.  This gives
\begin{align}
I = \int_{D'} g(z)\varphi(z) dz +  \int_{D}g(z)\partial_{\bar z}^{\mu*}\bigg (\rho(z)\int_{D'} \varphi(w)S(w,z) dw\bigg) dz. \label{Ifinal}
\end{align}
So returning to \eqref{visu1},
\begin{align*}
\int_{D'}\varphi(w)(\tilde{g}(w) - g(w))dw &=   \int_{D}g(z)\partial_{\bar z}^{\mu*}\bigg (\rho(z)\int_{D'} \varphi(w)S(w,z) dw\bigg) dz+ II.\\
\intertext{Letting $\psi(z) = \rho(z)\int_{D'}\varphi(w)S(w,z)dw$ yields}
&= \int_{D}g(z) \partial_{\bar z}^{\mu*}\psi(z) dz + \int_{D} F(z)\psi(z) dz.
\end{align*}
By \eqref{boundi} and \eqref{boundii}, $\psi$ and $\partial_{\bar z}^{\mu*} \psi$ are continuous with compact support and therefore can be approximated by smooth functions with compact support in $D'$.  Since $\partial_{\bar z}^\mu g = F$ in distribution, we see that
\begin{align*}
\int_{D'}\varphi(w)(\tilde{g}(w) - g(w))dw = 0.
\end{align*}
So $\tilde{g}(w) = g(w)$ a.e.\ in $D'$ and
\begin{align}
g(w) = \int_D g(z) \partial_{\bar z}^{\mu*}(\rho(z)S(w,z)) dz + \int_D F(z) \rho(z)S(w,z)dz. \label{finalform}
\end{align}
for a.e.\ $w \in D'$.
\end{proof}

\begin{theorem}
Let $g \in L^1_{\text{loc}}(\bb C)$, $Dg \in L^2(\bb C)$, $\mu, F \in C_c^\infty(\bb C)$ and $|\mu|\le k < 1$.  If $g$ satisfies
\begin{align*}
\partial_{\bar z}^\mu g = F
\end{align*}
in distribution, then $g$ is continuous.
\end{theorem}
\begin{proof}
By \eqref{boundi}, and \eqref{boundii},
\begin{align}
|g(w)| \le c_1\int_D \frac{|g(z)|}{|w-z|} dz + c_2. \label{rieszbound}
\end{align}
We already know that $Dg \in L^2(\bb C)$.  So by the Poincar\'e inequality, $g \in L^2_{\text{loc}}(\bb C)$ and therefore $g \in L^{3/2}_{\text{loc}}(\bb C)$.  Equation \eqref{rieszbound} shows that
\begin{align*}
|g(w)| \le c_1 I_1(\mathbbm{1}_{D}|g|)(w)| + c_2,
\end{align*}
where $I_1(|g|\mathbbm{1}_{D})(w) = \int_{\bb C} \frac{\mathbbm{1}_D(z)|g(z)|}{|w-z|}dz$. By Lemma \ref{riesz}, $I_1$ is a map from $L^p(\bb C)$ to $L^{2p/(2-p)}(\bb C)$. 

If we choose $p = 3/2$, then we see that $g \in L^6_{\text{loc}}(\bb C)$.
Both $\rho(z)S(w,z)$ and $\partial_{\bar z}^{\mu*}(\rho(z)S(w,z))$ satisfy \eqref{kboundi} and \eqref{kboundii}.  So by Lemma \ref{kernelcont}, $g$ is H\"older continuous with exponent $2/3$ (technically will agree a.e.\ with a H\"older continuous function).  That is, $|g(w_1) - g(w_2)| \le C|w_1-w_2|^{2/3}$.  By the proof of Lemma \ref{kernelcont}, $C$ only depends on $\mu, F$ and the $L^6$-norm of $g$ on the support of $\mu$.
\end{proof}
If we set $F = 0$, then the theorem states that the solution to the Beltrami equation with $\mu \in C_c^\infty(\bb C)$ is continuous.

\section{Isothermal Coordinates}\label{iso}
In this final section we have included the connection of the above discussion to the existence of isothermal coordinates on a manifold.  We also present a different proof for the existence of isothermal coordinates when $\mu$ is assumed to be H\"older continuous.  This proof has the benefit of being relatively simple but requires extra regularity for $\mu$.  Also, it only provides local coordinates not global solutions.  A version of this proof is presented in \cite{chern}.
\begin{prop}
Let $M$ be a 2-dimensional Riemannian manifold.  Then finding isothermal coordinates on $M$ is equivalent to solving $\Delta f = 0$. 
\end{prop}
\begin{proof}
Let $U$ be an open neighborhood of $p \in M$ and $f \colon U \to V \subset \bb C$ be a diffeomorphism.  Let $g_E$ be the Euclidean metric and $\lambda \colon \bb C \to (0,\infty)$ be $C^\infty$-smooth.  If $f^*g_E = \lambda g$, then $f$ defines isothermal coordinates near $p$.

On the other hand, let $f = (f_1,f_2)$ and suppose $*df_1 = df_2$.  Then
 \begin{align*}
f^*(g_E) &= (f^*dx)^2 + (f^*dy)^2,\\
f^*dx &= df_1 = \frac{\partial f_1}{\partial u}du + \frac{\partial f_1}{\partial v} dv\\
\intertext{and}
f^*dy &= df_2 = \frac{\partial f_2}{\partial u}du + \frac{\partial f_2}{\partial v} dv.\\
\intertext{So}
(f^*dx)^2 + (f^*dy)^2 &= (df_1)^2 + (df_2)^2 = (df_1)^2 + (*df_1)^2 \\
&= \frac{\partial f_1}{\partial u}^2du^2 + 2\frac{\partial f_1}{\partial u}\frac{\partial f_1}{\partial v}dudv + \frac{\partial f_1}{\partial v}^2dv^2 + \frac{\partial f_1}{\partial u}dv^2 - 2\frac{\partial f_1}{\partial u}\frac{\partial f_1}{\partial v}dudv + \frac{\partial f_1}{\partial v}^2 du^2\\
&= \bigg (\frac{\partial f_1}{\partial u}^2 + \frac{\partial f_1}{\partial v}^2\bigg)(du^2+dv^2).
\end{align*}
This shows that $f$ defines isothermal coordinates.  Therefore, the existence of isothermal coordinates is equivalent to the existence of $f_1,f_2$ such that $*df_1 = df_2$.  By the Poincare lemma,  if $d*df_1 = 0$, then $f_2$ exists in a convex neighborhood of $p$.  The Laplacian, $\Delta =  \delta d$, on functions can be expressed as
\begin{align*}
\Delta f_1 = - * d( * df_1).
\end{align*}
So solving for $f_1$ is equivalent to solving
\begin{align*}
\Delta f_1 = 0
\end{align*}
with $df_1 \ne 0$ (so that $f = (f_1,f_2)$ is a diffeomorphism).  This can be done at least for a small neighborhood of $p$ (see below or \cite[Ch. 5.11]{taylor}).
\end{proof}
As a note, the above was computed using real coordinates.  If we are interested in computing this in complex coordinates, then $*df_1 = df_2$ becomes 
\begin{align*}
(*-i)df = 0.
\end{align*}
Here, $f$ is a function of $z$ and $\bar z$.  We also have that $*dz = -idz$ and $*d\bar z = id\bar z$.  So if $*df = -idf$, then $\partial f/\partial \bar z = 0$.  If $d*df = 0$, then there exists, locally, $g$ such that $dg = *df$.  So
\begin{align*}
*d(g-if) = *(dg)-i*(df) = **(df)-i(dg) = -df-idg= -id(g-if).
\end{align*}
This shows the proof in the complex coordinate case. 

The above discussion shows that Theorem \ref{mainthm} provides isothermal coordinate charts for any Riemannian metric defined on all of $\bb C$.  If one only desires local solutions there is an easier way to prove this fact.  This method works if $\mu$ is assumed to be H\"older continuous.

Define the following two integral operators on $f \colon \bb C \to \bb C$:
\begin{align*}
Tf(z) = \frac{1}{\pi}\int_{\bb C} \frac{f(w)}{z-w} dw
\end{align*}
and
\begin{align*}
Hf(z) = \lim_{\epsilon \to 0}\frac{1}{\pi} \int_{\bb C \setminus B(z,\epsilon)} \frac{f(w)}{(w-z)^2}dw.
\end{align*}
The relevant properties are:
\begin{itemize}

\item[(1)] If $f \in L^p(D)$, $2 < p < \infty$, then $Tf$ is H\"older continuous with exponent $1-2/p$.

\item[(2)] $H$ is an isometry from $L^2(\bb C) \to L^2(\bb C)$.

\item[(3)] If $f \in L^p(D)$, $2 < p < \infty$, and $f$ has compact support in $D$, then $(Tf)_{\bar z} = f$ and $(Tf)_z = Hf$.  So $f \in W^{1,2}(D)$.

\end{itemize}
Property (1) follows from Lemma \ref{kernelcont}.  Properties (2) and (3) are true for $f \in C_c^2(\bb C)$ and can be shown in general for $f \in L^p(\bb C)$ by taking an approximating sequence.

\begin{lemma}\label{holderH}
Suppose $f\colon\bb C \to \bb C$ is H\"older continuous with exponent $\alpha$, $0 < \alpha < 1$ and $f$ has compact support in $B(0,R)$.  Then $Hf$ is H\"older continuous with exponent $\alpha$.  Specifically,
\begin{align*}
|Hf(z_1) - Hf(z_2)| \le C_\alpha A|z_1 - z_2|^\alpha,
\end{align*}
where $A$ is the H\"older constant for $f$ and $C_\alpha = \frac{1}{\alpha}2^\alpha + \frac{1}{\alpha}3^\alpha + \frac{1}{1-\alpha}2^{\alpha - 1}+1$.
\end{lemma}

\begin{proof}
Let $D = B(0,R)$, $z_1,z_2 \in D$ and $r = |z_1-z_2|$.  Let $D' = B(z_1,2r)$.  Then
\begin{align*}
|Hf(z_1) - Hf(z_2)| =\frac{1}{\pi}\bigg | \int_{D'}\frac{f(w)-f(z_1)}{(w-z_1)^2}dw &- \int_{D'} \frac{f(w)-f(z_2)}{(w-z_2)^2}dw \\ &+ \int_{D\setminus D'}\frac{f(w) - f(z_1)}{(w-z_1)^2}dw - \int_{D\setminus D'} \frac{f(w) - f(z_1)}{(w-z_1)^2} dw\bigg |.
\end{align*} 
The first two terms are dealt with similarly.  For the first term the bound is
\begin{align*}
\bigg |\int_{D'}\frac{f(w)-f(z_1)}{(w-z_1)^2}dw\bigg | &\le \int_{D'} \frac{A|w-z_1|^\alpha}{|w-z_1|^2} dw\\
&= A\int_0^{2r}s^{\alpha - 1} ds = \frac{A}{\alpha}2^\alpha r^\alpha.
\end{align*}
The second term has a bound of $\frac{A}{\alpha}3^\alpha r^\alpha$, which can be seen by integrating on the larger domain, $B(z_2,3r)$.

The third and fourth terms can be expressed as:
\begin{align*}
\int_{D\setminus D'} (f(w&)-f(z_1)) 2\int_{z_1}^{z_2}\frac{1}{(w-\zeta)^3} d\zeta dw\\
&=\int_{z_1}^{z_2}\int_{D\setminus D'}\frac{(f(w) - f(\zeta))}{(w-\zeta)^3} + \frac{f(\zeta)-f(z_1)}{(w-\zeta)^3} dwd\zeta
\end{align*}
by Fubini's theorem.
The first term:
\begin{align*}
\bigg |\int_{D\setminus D'}\frac{(f(w) - f(\zeta))}{(w-\zeta)^3}dw\bigg | & \le \int_{2r}^MA s^{\alpha - 2} ds \le \frac{A}{1-\alpha} (2r)^{\alpha- 1}.
\intertext{The second term:}
\bigg |\int_{D\setminus D'}\frac{(f(\zeta) - f(z_1))}{(w-\zeta)^3}dw\bigg | &\le \int_{2r}^M A r^\alpha s^{-2} ds \le Ar^{\alpha - 1}.
\end{align*}
When we integrate over the $\zeta$ variable the inequalities sum to $(\frac{A}{1-\alpha}2^{\alpha - 1}+A)r^\alpha$.  So combining this with the previous bounds we see that $Hf$ is H\"older continuous with exponent $\alpha$ and coefficient $C_{A,\alpha} = \frac{A}{\alpha}2^\alpha + \frac{A}{\alpha}3^\alpha + \frac{A}{1-\alpha}2^{\alpha - 1}+A$.
\end{proof}

\begin{theorem}\label{isothermlocal}
Suppose $\mu \in C_c(\bb C)$ is H\"older continuous with exponent $\alpha$, H\"older norm $A$, and $\|\mu\|_\infty = k < 1$.  In addition, suppose $\mu$ has compact support in $B(0,R)$.  If $R$ is sufficiently small, then there exists a map $f\colon B(0,R) \to \bb C$ such that $f_{\bar z} = \mu f_z$. Furthermore, $f$ is a homeomorphism onto its image.
\end{theorem}
\begin{proof}
Without loss of generality we may assume that $\mu(0) = 0$. 
Since $\mu$ is H\"older continuous, by Lemma \ref{holderH}, $H\mu$ is H\"older continuous as well.  Consider
\begin{align*}
h \coloneqq \sum_{n = 1}^\infty (H\mu)^n = H\mu + H\mu H\mu + \cdots.
\end{align*}
Then, for $z \in B(0,R)$,
\begin{align*}
|(H\mu)^n(z)| \le C_{\alpha,A}^nR^{\alpha n}.
\end{align*}
To see why this is true, take $f$ H\"older continuous with exponent $\alpha$ and constant $C_f$.  Then, by Lemma \ref{holderH},
\begin{align*}
|H\mu f(z_1) - H\mu f(z_2)| \le AR^\alpha C_\alpha C_f |z_1 - z_2|^\alpha.
\end{align*}
If $f = H\mu$, then
\begin{align*}
|H\mu f(z)| \le C_\alpha^2 A^2R^{2\alpha}.
\end{align*}
By induction we get the claim above.
So for $R$ sufficiently small, this is less than $1$ and the sum converges uniformly.  Thus, $h$ is also H\"older continuous.  

Note that $h$ satisfies
\begin{align*}
h - H(\mu h) = H\mu.
\end{align*}
Define
\begin{align*}
f(z) = z + T(\mu + \mu h).
\end{align*}
Since $\mu + \mu h$ is continuous, $\mu + \mu h \in L^p(D)$.  By property (3), $f \in C^1(D)$ and $1 +Hf = f_z $ is continuous.  In addition,
\begin{align*}
f_{\bar z} = \mu + \mu h = \mu(1+ h) = \mu(1 + H\mu + H(\mu h)) = \mu f_z.
\end{align*}
So $f$ solves \eqref{beq}.  To see that $f$ is injective, consider its Jacobian,
\begin{align*}
J_f = |f_z|^2 - |f_{\bar z}|^2 = |f_z|^2(1-|\mu|^2) \ge |f_z|^2(1-k^2).
\end{align*}
And
\begin{align*}
|f_z| = |1 + H(\mu + \mu h)| \ge 1 - C R^{\alpha}.
\end{align*}
So for small $R$, $|f_z| > 0$ and $f$ is locally injective.  By the inverse function theorem we can choose an $R$ so that $f$ is a homeomorphism onto its image.

\end{proof}

\end{document}